\documentclass{amsart}

\usepackage{macros}

\standardsettings

\newcommand{\xmod}{a}
\newcommand{\ymod}{b}
\newcommand{\zmod}{c}

\draftfalse

\begin{document}
\title[Interpreting Patterson--Sullivan measures as Hausdorff and packing measures]{On interpreting Patterson--Sullivan measures of geometrically finite groups as Hausdorff and packing measures}

\authordavid

%\subjclass[2010]{Primary }
%\date{}
%\dedicatory{}

\begin{abstract}
We provide a new proof of a theorem whose proof was sketched by Sullivan ('82), namely that if the Poincar\'e exponent of a geometrically finite Kleinian group $G$ is strictly between its minimal and maximal cusp ranks, then the Patterson--Sullivan measure of $G$ is not proportional to the Hausdorff or packing measure of any gauge function. This disproves a conjecture of Stratmann ('97, '06).
\end{abstract}
\maketitle

\section{Introduction}
Let $G$ be a nonelementary geometrically finite Kleinian group with at least one cusp.\footnote{In this paper, a \emph{Kleinian group} is a discrete subgroup of $\Isom(\B^d)$ for any $d\geq 2$, where $\B^d$ denotes $d$-dimensional hyperbolic space.} Let $\kmin$ and $\kmax$ denote the smallest and largest cusp ranks, respectively, and let $\delta$ denote the Poincar\'e exponent of $G$. Let $\Lambda$ be the limit set of $G$, and let $\mu$ be the Patterson--Sullivan measure of $G$, i.e. the unique probability measure on $\Lambda$ which is $\delta$-conformal with respect to $G$. It is known \cite[Theorem 2]{Sullivan_entropy} that
\begin{align*}
\mu \propto \HH^\delta\given\Lambda &\;\;\Leftrightarrow\;\; \delta \geq \kmax \text{ and}\\
\mu \propto \PP^\delta\given\Lambda &\;\;\Leftrightarrow\;\; \delta \leq \kmin.
\end{align*}
Here $\HH^\delta$ and $\PP^\delta$ denote the Hausdorff and packing measures, respectively, in dimension $\delta$, and $\propto$ is the proportionality symbol. The above conditions can be thought of as giving a ``geometric interpretation'' of the Patterson--Sullivan measure $\mu$ in the cases $\delta\geq \kmax$ and $\delta\leq \kmin$, since the measures $\HH^\delta\given\Lambda$ and $\PP^\delta\given\Lambda$ are defined using only the metric structure of $\Lambda$, without reference to the group $G$.

%A natural question is whether the Patterson--Sullivan measure can still be given a geometrical interpretation when it is no longer proportional to either Hausdorff or packing measure, i.e. when $\kmin < \delta < \kmax$. For example,
%This question was addressed in the dissertation of Ala-Mattila, a student of Tukia \cite{Ala-Mattila}.

If $\kmin < \delta < \kmax$, then the above geometric interpretations fail, but it is natural to ask whether $\mu$ is proportional to $\HH^\psi\given\Lambda$ or $\PP^\psi\given\Lambda$ for some Hausdorff gauge function $\psi$ (cf. \6\ref{subsectionRTT}). In fact, this question has a confusing history: Sullivan originally asked it in 1982 \cite[p.71]{Sullivan_discrete_conformal_groups} and then sketched a proof of a negative answer later that year \cite[second Corollary on p.235]{Sullivan_disjoint_spheres} (see also \cite[Remark (2) on p.261]{Sullivan_entropy}). Over a decade later, and without referencing Sullivan's sketch, Stratmann stated that the question was open on two separate occasions \cite[p.68]{Stratmann5}, \cite[p.235]{Stratmann6}, even conjecturing that the answer is positive for the gauge function
\begin{equation}
\label{conjecturedphi}
\psi(r) = r^\delta \exp\left(\frac{\kmax - \delta}{2\delta - \kmax}\left(\log\log\frac{1}{r} + \log\log\log\log\frac{1}{r}\right)\right).\footnote{Since Stratmann was working in $\B^3$ only, in his formula, $\kmin = 1$ and $\kmax = 2$.}
\end{equation}
(Cf. \cite{BishopJones2}, where a similar result is proven for analytically finite but geometrically infinite groups.) In 2011 Ala-Mattila, a doctoral student of Tukia, noticed these inconsistencies in the literature in his dissertation \cite[p.105]{Ala-Mattila}, and stated that he was not able to follow Sullivan's proof. (Nevertheless, his thesis answers a different but related question; namely, it provides a dynamics-independent construction of the Patterson--Sullivan measure, and thus a ``geometric interpretation'' of the Patterson--Sullivan measure in the sense described above.)

The author of the present paper confesses that he is also not able to follow Sullivan's proof, so based on the above, it would seem that the question is still open. The purpose of this note is to prove using different methods that Sullivan's answer is correct. Specifically, we prove the following theorem:

\begin{theorem}
\label{theorem1}
Let $G$ be a nonelementary geometrically finite Kleinian group with at least one cusp, let $\psi$ be a Hausdorff gauge function, and let $\Psi:\R\to\R$ be the unique function satisfying
\begin{equation}
\label{Psidef}
\psi(r) = r^\delta \exp\left(\Psi\left(\log\frac{1}{r}\right)\right),
\end{equation}
where $\delta$ is the Poincar\'e exponent. Assume that
\begin{equation}
\label{Psiassumption}
\text{$\Psi$ is eventually differentiable and monotonic, and the limit $\lim_{t\to\infty}\Psi'(t)$ exists.}
\end{equation}
Let $\kmin$ and $\kmax$ denote the smallest and largest cusp ranks, respectively. Then
\begin{itemize}
\item[(i)] If $\delta < \kmax$, then $\HH^\psi(\mu) = 0$ or $\infty$, according to whether the series
\begin{equation}
\label{Hausdorffseries}
\sum_{t = 1}^\infty \exp\left(-\frac{2\delta - \kmax}{\kmax - \delta}\Psi(t)\right)
\end{equation}
diverges or converges, respectively.
\item[(ii)] If $\delta > \kmin$, then $\PP^\psi(\mu) = 0$ or $\infty$, according to whether the series
\begin{equation}
\label{packingseries}
\sum_{t = 1}^\infty \exp\left(\frac{2\delta - \kmin}{\delta - \kmin}\Psi(t)\right)
\end{equation}
converges or diverges, respectively.
\end{itemize}
In particular, if $\kmin < \delta < \kmax$ then neither $\HH^\psi(\mu)$ nor $\PP^\psi(\mu)$ is positive and finite for any Hausdorff gauge function $\psi$ satisfying \eqref{Psiassumption}, so the Patterson--Sullivan measure cannot be interpreted geometrically as a Hausdorff or packing measure of such a gauge function. (See also Corollary \ref{corollary1} below.)
\end{theorem}

\begin{example}
For the value of $\psi$ given in \eqref{conjecturedphi}, we have
\[
\Psi(t) = \frac{\kmax - \delta}{2\delta - \kmax}\big(\log(t) + \log\log\log(t)\big),
\]
so the series \eqref{Hausdorffseries} reduces to
\[
\sum_{t = 1}^\infty \frac{1}{t\log\log(t)},
\]
which diverges. Therefore $\HH^\psi(\mu) = 0$ for this $\psi$.
\end{example}

\begin{remark}
The relations in Theorem \ref{theorem1} can be a little disorienting at first -- divergence of the series \eqref{Hausdorffseries} guarantees $\HH^\psi(\mu) = 0$, while divergence of the series \eqref{packingseries} guarantees $\PP^\psi(\mu) = \infty$. The unifying insight is that higher values of $\Psi$ correspond to higher (decreasing more slowly as $r\searrow 0$) values of $\psi$, which in turn correspond to higher values of $\HH^\psi(\mu)$ and $\PP^\psi(\mu)$. Since the series \eqref{Hausdorffseries} is decreasing with respect to $\Psi$ while \eqref{packingseries} is increasing with respect to $\Psi$, divergence of \eqref{Hausdorffseries} indicates low (not close to $+\infty$) values of $\Psi$ while divergence of \eqref{packingseries} indicates high (not close to $-\infty$) values of $\Psi$.
\end{remark}

\begin{remark}
Theorem \ref{theorem1} can be generalized to the setting of pinched Hadamard manifolds,\footnote{In this setting, the cusp rank $k_\xi$ of a parabolic point $\xi$ is interpreted to be the number $2\delta(G_\xi)$, where $G_\xi$ is the stabilizer of $\xi$ in $G$ and $\delta(G_\xi)$ is the Poincar\'e exponent of $G_\xi$; cf. \cite[Th\'eor\`eme 3.2]{Schapira}.} under some additional assumptions regarding the group $G$. Indeed, the three main results used in the proof of Theorem \ref{theorem1}, namely the Rogers--Taylor--Tricot density theorem, the Global Measure Formula, and Stratmann--Velani's Khinchin-type Theorem for Geometrically Finite Groups, are generalized to this setting in \cite[\68]{MSU},\footnote{The ``arbitrary metric space'' of \cite[\68]{MSU} should be interpreted to be the Gromov boundary of the pinched Hadamard manifold in question, endowed with a visual metric.} \cite[Th\'eor\`eme 7.2]{Schapira}, and \cite[Theorem 3]{HP_Khinchin}, respectively, and the deduction of Theorem \ref{theorem1} from these theorems does not make use of constant curvature in any essential way. However, for simplicity of exposition we stick to the case of standard hyperbolic space.
\end{remark}

%\begin{remark}
%If $\psi$ is a Hardy $L$-function (i.e. $\psi$ can be written in terms of the standard algebraic operations, exponents, and logarithms), then the assumption \eqref{Psiassumption} is satisfied.
%\end{remark}

\begin{corollary}
\label{corollary1}
Let $G$ and $\psi$ be as in Theorem \ref{theorem1}.
\begin{itemize}
\item[(i)] If $\delta < \kmax$, then $\HH^\psi(\Lambda) = 0$ or $\infty$.
\item[(ii)] If $\delta > \kmin$, then $\PP^\psi(\Lambda) = 0$ or $\infty$.
\end{itemize}
\end{corollary}
\begin{proof}
Without loss of generality we may assume that
\begin{equation}
\label{Psiassumption2}
\lim_{t\to\infty} \Psi'(t) = 0.
\end{equation}
Indeed, since  the Hausdorff and packing dimensions of $\Lambda$ are both equal to $\delta$ \cite{Sullivan_entropy,StratmannUrbanski4}, if \eqref{Psiassumption2} fails then the corollary follows from comparing the measures $\HH^\psi$ and $\PP^\psi$ with the measures $\HH^s$ and $\PP^s$, where $s$ lies between $\delta$ and $\delta + \lim_{t\to\infty} \Psi'(t)$.

Suppose $\delta < \kmax$ but $0 < \HH^\psi(\Lambda) < \infty$. Then \eqref{Psiassumption2} implies that the measure $\nu = \HH^\psi\given\Lambda$ is $\delta$-conformal with respect to $G$ (Lemma \ref{lemmaHPconformal}). Next, the uniqueness of $\delta$-conformal measures for geometrically finite groups (e.g. \cite[Theorem 1]{Sullivan_entropy}) implies that $\nu = \alpha\mu$ for some $\alpha > 0$. But then $\HH^\psi(\mu) = \HH^\psi(\nu) = \HH^\psi(\Lambda) \in (0,\infty)$, contradicting Theorem \ref{theorem1}. This demonstrates (i), and (ii) is proven similarly.
\end{proof}

Note however that this argument does not tell us that the value of $\HH^\psi(\Lambda)$ is determined by the convergence or divergence of the series \eqref{Hausdorffseries}; \emph{a priori} it may be the case that $\HH^\psi(\Lambda) = \infty$ while $\HH^\psi(\mu) = 0$. An interesting question is whether or not this is possible.

{\bf Convention.} The symbols $\lesssim$, $\gtrsim$, and $\asymp$ will denote coarse asymptotics; a subscript of $\plus$ indicates that the asymptotic is additive, and a subscript of $\times$ indicates that it is multiplicative. For example, $A\lesssim_\times B$ means that there exists a constant $C > 0$ (the \emph{implied constant}) such that $A\leq C B$. It is understood that the implied constant $C$ is only allowed to depend on certain ``universal'' parameters, to be understood from context.\\

{\bf Acknowledgements.} The author thanks Tushar Das for calling his attention to Ala-Mattila's thesis and to the fact that the question addressed in this paper is still considered open, as well as for helpful comments on an earlier version of this paper. The author also thanks the referee for valuable comments.

\section{Preliminaries}
Our theorem is essentially a combination of three known theorems: the Rogers--Taylor--Tricot density theorem, the Global Measure Formula, and Stratmann--Velani's Khinchin-type Theorem for Geometrically Finite Groups. In this section we recall these theorems and the definitions used in them.

\subsection{The Rogers--Taylor--Tricot density theorem}
\label{subsectionRTT}
We first recall the basics of generalized Hausdorff and packing measures. Let $\psi:\Rplus\to\Rplus$ be a \emph{Hausdorff gauge function}, i.e. a positive increasing continuous function satisfying $\psi(0) = 0$. The \emph{$\psi$-dimensional Hausdorff measure} of a set $A\subset\R^d$ is
\[
\HH^\psi(A) = \lim_{\epsilon\searrow 0}\inf\left\{\sum_{i = 1}^\infty \psi(\diam(U_i)): \text{$(U_i)_1^\infty$ is a countable cover of $A$ with $\diam(U_i)\leq\epsilon \all i$}\right\}
\]
and the  \emph{$\psi$-dimensional packing measure} of $A$ is defined by the formulas
\[
\w\PP^\psi(A) = \lim_{\epsilon\searrow 0}\sup\left\{\sum_{j = 1}^\infty \psi(\diam(B_j)):
\begin{split}
&\text{$(B_j)_1^\infty$ is a countable disjoint collection of balls}\\
&\text{with centers in $A$ and with $\diam(B_j)\leq\epsilon \all j$}
\end{split}\right\}
\]
\[
\PP^\psi(A) = \inf\left\{ \sum_{i = 1}^\infty \w\PP^\psi(A_i) : A \subset \bigcup_{i = 1}^\infty A_i\right\}.
\]
A special case is when $\psi(r) = r^s$ for some $s > 0$, in which case the shorthands $\HH^\psi = \HH^s$ and $\PP^\psi = \PP^s$ are used.

If $\mu$ is a Borel measure on $\R^d$, then we let
\begin{align*}
\HH^\psi(\mu) &= \inf\left\{\HH^\psi(A): \mu(\R^d\butnot A) = 0\right\},\\
\PP^\psi(\mu) &= \inf\left\{\PP^\psi(A): \mu(\R^d\butnot A) = 0\right\}.
\end{align*}
Moreover, for each point $x\in\R^d$ let
\begin{align*}
\overline D_\mu^\psi(x) &= \limsup_{r\searrow 0} \frac{\mu(B(x,r))}{\psi(r)},\\
\underline D_\mu^\psi(x) &= \liminf_{r\searrow 0} \frac{\mu(B(x,r))}{\psi(r)}\cdot
\end{align*}

\begin{theorem}[Rogers--Taylor--Tricot density theorem]
\label{theoremRTT}
Fix $d\in\N$, let $\mu$ be a Borel measure on $\R^d$, and let $\psi$ be a Hausdorff gauge function. Then for every Borel set $A\subset\R^d$,
\begin{align} \label{rogerstaylor}
\mu(A) \inf_{x\in A}\frac{1}{\overline D_\mu^\psi(x)} \lesssim_\times \HH^\psi(A) &\lesssim_\times \mu(\R^d) \sup_{x\in A}\frac{1}{\overline D_\mu^\psi(x)}\\ \label{taylortricot}
\mu(A) \inf_{x\in A}\frac{1}{\underline D_\mu^\psi(x)} \lesssim_\times \PP^\psi(A) &\lesssim_\times \mu(\R^d) \sup_{x\in A}\frac{1}{\underline D_\mu^\psi(x)}\cdot
\end{align}
\end{theorem}
\noindent Formula \eqref{rogerstaylor} was proven by Rogers and Taylor in 1959 \cite{RogersTaylor}, while formula \eqref{taylortricot} was proven by Taylor and Tricot in 1985 \cite[Theorem 2.1]{TaylorTricot}. A short proof of \eqref{rogerstaylor} can be found in \cite[Proposition 3.4 of Chapter 4]{BishopPeres}. A generalization to arbitrary metric spaces can be found in \cite[\68]{MSU}.

\begin{corollary}
\label{corollaryRTT}
With $d$, $\mu$, and $\psi$ as above,
\begin{align} \label{rogerstaylorv2}
\mu(\R^d) \essinf_{x\in\R^d} \frac{1}{\overline D_\mu^\psi(x)} \lesssim_\times \HH^\psi(\mu) &\lesssim_\times \mu(\R^d) \esssup_{x\in\R^d} \frac{1}{\overline D_\mu^\psi(x)}\\ \label{taylortricotv2}
\mu(\R^d) \essinf_{x\in\R^d} \frac{1}{\underline D_\mu^\psi(x)} \lesssim_\times \PP^\psi(\mu) &\lesssim_\times \mu(\R^d) \esssup_{x\in\R^d} \frac{1}{\underline D_\mu^\psi(x)}\cdot
\end{align}
\end{corollary}
\begin{proof}
Let
\[
A_0 = \left\{x\in\R^d : \essinf \overline D_\mu^\psi \leq \overline D_\mu^\psi(x) \leq \esssup \overline D_\mu^\psi\right\}.
\]
Since $\mu(\R^d\butnot A_0) = 0$, we have
\[
\HH^\psi(\mu) \leq \HH^\psi(A_0) \lesssim_\times \mu(\R^d) \esssup_{x\in\R^d} \frac{1}{\overline D_\mu^\psi(x)}\cdot
\]
On the other hand, given any set $B\subset\R^d$ such that $\mu(\R^d\butnot B) = 0$, applying \eqref{rogerstaylor} with $A = A_0\cap B$ gives
\[
\HH^\psi(B) \gtrsim_\times \mu(A_0\cap B) \essinf_{x\in\R^d}\frac{1}{\overline D_\mu^\psi(x)} = \mu(\R^d) \essinf_{x\in\R^d}\frac{1}{\overline D_\mu^\psi(x)},
\]
and taking the infimum over all such $B$s demonstrates \eqref{rogerstaylorv2}. The proof of \eqref{taylortricotv2} proceeds in the same manner.
\end{proof}

%\begin{proof}
%Let $A_1$ and $A_2$ be the set of all $x$ for which \eqref{A1} and \eqref{A2} hold, respectively. Then by Theorem \ref{theoremRTT}(ii), $\HH^\psi(\mu) \leq \HH^\psi(A_2) \leq k_d \theta^{-1} \mu(\R^d)$. Conversely, given any set $A$ for which $\mu(\R^d\butnot A) = 0$, Theorem \ref{theoremRTT}(i) gives $\HH^\psi(A) \geq \HH^\psi(A\cap A_1) \geq \alpha^{-1}\mu(A\cap A_1) = \alpha^{-1}\mu(\R^d)$; taking the infimum gives $\HH^\psi(\mu) \geq \alpha^{-1}\mu(\R^d)$.
%\end{proof}

\subsection{Geometrically finite groups}
Fix $d\geq 2$, let $\B^d$ denote the $d$-dimensional Poincar\'e disk, and let $G$ be a nonelementary geometrically finite Kleinian group acting on $\B^d$ with at least one cusp. Recall that this means that there exists a disjoint $G$-invariant collection of horoballs $\scrH$ with the property that the quotient $\big(\CC_G\butnot\bigcup(\scrH)\big)/G$ is compact, where $\CC_G$ is the convex hull of the limit set of $G$ (e.g. \cite[Definition (GF1)]{Bowditch_geometrical_finiteness}). The elements of $\scrH$ are centered at the parabolic fixed points of $G$. For each parabolic point $\xi$, let $H_\xi$ denote the unique element of $\scrH$ centered at $\xi$, and let $k_\xi$ denote the rank of $\xi$.

Let $\delta$ denote the \emph{Poincar\'e exponent} of $G$, i.e.
\[
\delta = \inf\left\{s\geq 0 : \sum_{g\in G} e^{-s\dist(0,g(0))} < \infty\right\},
\]
where $\dist$ denotes the hyperbolic distance on $\B^d$. Let $\Lambda$ denote the \emph{limit set} of $G$, i.e. the set
\[
\Lambda = \left\{\xi\in\del\B^d : \xi = \lim_{n\to \infty} g_n(0) \text{ for some sequence $(g_n)_1^\infty$ in $G$}\right\}.
\]
Finally, let $\mu$ denote the \emph{Patterson--Sullivan measure} of $G$, i.e. the unique Borel probability measure on $\Lambda$ obeying the transformation rule
\[
\mu(g(A)) = \int_A |g'(\xi)|^\delta \dee\mu(\xi) \all A\subset\del\B^d \all g\in G.
\]

\subsubsection{The global measure formula}
Let the functions $k,b:\B^d\to\Rplus$ be defined as follows:
\begin{itemize}
\item If $x\in \B^d\butnot\bigcup(\scrH)$, then $k(x) = b(x) = 0$.
\item If $x\in H_\xi$ for some $H_\xi\in \scrH$, then $k(x) = k_\xi$ and $b(x) = \dist(x,\del H_\xi)$.
\end{itemize}
Given a point $\eta\in\Lambda$, let $\eta_t$ denote the unique point on the geodesic ray connecting $0$ and $\eta$ whose hyperbolic distance from $0$ is equal to $t$. 

\begin{theorem}[Global measure formula, {\cite[Theorem 2]{StratmannVelani}}]
\label{theoremglobalmeasure}
For any $\eta\in\Lambda$ and for any $t > 0$,
\[
\mu(B(\eta,e^{-t})) \asymp_\times e^{-\delta t} e^{b(\eta_t)(k(\eta_t) - \delta)}.
\]
\end{theorem}
%There is a convenient way of reformulating \eqref{globalmeasureformula}. For each parabolic point $\xi$ of $G$, let $k_\xi$ denote the rank of $\xi$, and let $b_\xi:\B^d\to\R$ be the unique horofunction based at $\xi$ which is zero on $\del H_\xi$ and positive in $H_\xi$, so that
%\[
%b_\xi(x) = d(x) \all x\in H_\xi.
%\]

\subsubsection{Khinchin-type theorem for geometrically finite groups}
For each $\xi\in\del\B^d$ and $r > 0$, let $H(\xi,r)$ be the unique horoball centered at $\xi$ with Euclidean radius $r$, i.e.
\[
H(\xi,r) := B\big((1 - r)\xi,r\big).
\]
If $\xi$ is a parabolic point, then we let $r_\xi$ denote the unique value such that $H(\xi,r_\xi)\in \scrH$, so that $H_\xi = H(\xi,r_\xi)$.

Let $\phi:(0,\infty)\to(0,\infty)$ be a \emph{Khinchin function}, i.e. a positive monotonic doubling function. Here, by saying that $\phi$ is \emph{doubling}, we mean that there exist constants $C_1,C_2 > 1$ such that
\[
\frac{1}{C_1} \leq \frac{y}{x} \leq C_1 \;\;\Rightarrow\;\; \frac{1}{C_2} \leq \frac{\phi(y)}{\phi(x)} \leq C_2 \all x,y > 0.
\]
(For $\phi$ to be doubling, it is sufficient but not necessary that the function $\log\circ\phi\circ\exp$ is uniformly continuous.) Now fix a parabolic point $p$, and consider the set
\[
\Omega_p(\phi) := \{\eta\in\Lambda: \|\xi - \eta\| \leq \phi(r_\xi)r_\xi \text{ for infinitely many $\xi\in G(p)$}\}.
\]
%\[
%\Omega_p(\phi) := \{\eta\in\Lambda: \eta\in\Pi(H(\xi,\phi(r_\xi)r_\xi)) \text{ for infinitely many $\xi\in G(p)$}\}.
%\]
%Here $\Pi:\B^d\butnot\{0\}\to S^{d - 1}$ is the radial projection map.

\begin{theorem}[Khinchin-type theorem for geometrically finite Kleinian groups, {\cite[Theorem 4]{StratmannVelani}}]
\label{theoremkhinchin}
$\mu(\Omega_p(\phi)) = 0$ or $1$ according to whether the series
\[
\sum_{n = 1}^\infty \phi(\lambda^n)^{2\delta - k_p}
\]
converges or diverges, respectively. Here $\lambda\in (0,1)$ is a constant depending only on $G$.
\end{theorem}

\section{Proof of Theorem \ref{theorem1}}
As in Theorem \ref{theorem1}, let $G$ be a nonelementary geometrically finite Kleinian group with at least one cusp, let $\psi$ be a Hausdorff gauge function, let $\Psi$ be given by \eqref{Psidef}, and assume that \eqref{Psiassumption} holds. As in the proof of Corollary \ref{corollary1}, we may without loss of generality assume that
\begin{equation}
\label{Psiassumption3}
\lim_{t\to\infty} \Psi'(t) = 0.% \text{ and } \lim_{t\to\infty} \Psi(t) = \begin{cases}
%+\infty & \text{ Case (i)}\\
%-\infty & \text{ Case (ii)}
%\end{cases}
\end{equation}
%as otherwise the theorem follows from comparing the measures $\HH^\psi$ and $\PP^\psi$ with the measures $\HH^s$ and $\PP^s$, where $s$ lies between $\delta$ and $\delta + \lim_{t\to\infty} \Psi'(t)$.
% $s = \delta + (1/2)\lim_{t\to\infty}\Psi'(t)$. Here ``Case (i)'' means that we are trying to prove case (i) of Theorem \ref{theorem1}, and ``Case (ii)'' means that we are trying to prove case (ii).%
Let $P$ be a complete set of inequivalent parabolic points of $G$, and let
\[
P_{>\delta} = \{p\in P : k_p > \delta\}, \;\; P_{<\delta} = \{p\in P : k_p < \delta\}.
\]
\begin{lemma}
\label{lemma1}
For each $p\in P_{<\delta}\cup P_{>\delta}$ and $\alpha > 0$ let $\psi_p,\theta_p,\phi_{p,\alpha}:(0,\infty)\to(0,\infty)$ be defined by
\begin{align*}
\psi_p(r) &= \exp\left(-\frac{\Psi(\log(1/r))}{k_p - \delta}\right), \;\;
\theta_p(r) = \frac{r}{\psi_p(r)}, \;\;
\phi_{p,\alpha}(r) = \frac{\theta_p^{-1}(r/\alpha)}{r}\cdot
\end{align*}
\begin{itemize}
\item[(i)] If $\delta < \kmax$ and $\lim_{t\to\infty} \Psi(t) = +\infty$, then for $\mu$-a.e. $\eta\in\Lambda$,
\[
\overline D_\mu^\psi(\eta) \asymp_\times \max_{p\in P_{>\delta}} \Big(\sup\{\alpha > 0: \eta\in \Omega_p(\phi_{p,\alpha})\}\Big)^{k_p - \delta}.
\]
\item[(ii)] If $\delta > \kmin$ and $\lim_{t\to\infty} \Psi(t) = -\infty$, then for $\mu$-a.e. $\eta\in\Lambda$,
\[
\underline D_\mu^\psi(\eta) \asymp_\times \min_{p\in P_{<\delta}} \Big(\sup\{\alpha > 0: \eta\in \Omega_p(\phi_{p,\alpha})\}\Big)^{k_p - \delta}.
\]
\end{itemize}
\end{lemma}
\begin{remark}
From the assumptions \eqref{Psiassumption} and \eqref{Psiassumption3}, we can see that
\begin{itemize}
\item[(i)] $\psi_p$ is monotonic in a neighborhood of 0 (since $\Psi$ is eventually monotonic),
\item[(ii)] $\theta_p$ is increasing in a neighborhood of 0 (since $\lim_{t\to-\infty} [\log\circ\theta_p\circ\exp]'(t) = 1$),
\item[(iii)] $\phi_{p,\alpha}$ is monotonic in a neighborhood of 0 (since $\phi_{p,\alpha}(r) = \frac1\alpha \psi_p\big(\theta_p^{-1}(r/\alpha)\big)$), and
\item[(iv)] $\phi_{p,\alpha}$ is doubling
\[
\text{(since $\lim_{t\to-\infty} [\log\circ\phi_{p,\alpha}\circ\exp]'(t) = \frac{\lim_{t\to-\infty} [\log\circ\psi_p\circ\exp]'(t)}{\lim_{t\to-\infty} [\log\circ\theta_p\circ\exp]'(t)} = \frac 01 = 0$).}
\]
\end{itemize}
In particular, $\phi_{p,\alpha}$ is a Khinchin function.
\end{remark}
\begin{proof}[Proof of Lemma \ref{lemma1}]
We prove only (ii); the proof of (i) is similar but easier. %\footnote{Indeed, the proof of (i) can be realized by making the following modifications to the current proof: change all minima, infima, and lim-infs to maxima, suprema, and lim-sups; replace $P_{<\delta}$ by $P_{>\delta}$; }
By Theorem \ref{theoremglobalmeasure},
\[
\underline D_\mu^\psi(\eta) = \liminf_{t\to\infty} \frac{\mu(B(\eta,e^{-t}))}{\psi(e^{-t})}
\asymp_\times \liminf_{t\to\infty} \frac{e^{-\delta t} e^{b(\eta_t)(k(\eta_t) - \delta)}}{e^{-\delta t} e^{\Psi(t)}}
= \exp\liminf_{t\to\infty} [b(\eta_t)(k(\eta_t) - \delta) - \Psi(t)].
\]
Since by assumption $\lim_{t\to\infty} \Psi(t) = -\infty$, values of $t$ for which $b(\eta_t)(k(\eta_t) - \delta) \geq 0$ will not affect the lim-inf. On the other hand, since $\delta > \kmin$, for $\mu$-a.e. $\eta\in\Lambda$ we have $b(\eta_t)(k(\eta_t) - \delta) < 0$ for a sequence of $t$ tending to infinity. Thus
\begin{equation}
\label{estimate}
\log\underline D_\mu^\psi(\eta) \asymp_\plus \liminf_{\substack{t\to\infty \\ b(\eta_t)(k(\eta_t) - \delta) < 0}} [b(\eta_t)(k(\eta_t) - \delta) - \Psi(t)] = \liminf_{\xi\in G(P_{<\delta})} \inf_{\substack{t > 0 \\ \eta_t\in H_\xi}} [b(\eta_t)(k_\xi - \delta) - \Psi(t)].
\end{equation}
Fix $\xi\in G(P_{<\delta})$. To estimate the infimum on the right hand side, we use the following estimate for $b(\eta_t)$ which is valid for $\eta_t\in H_\xi$ (cf. Lemma \ref{lemmadetat} below):
\begin{equation}
\label{detat}
b(\eta_t) \asymp_\plus f_\xi(t) := \min\left(t - \log\frac{1}{r_\xi}, 2\log\frac{1}{\|\xi - \eta\|} - \log\frac{1}{r_\xi} - t\right).
\end{equation}
Let $t_\xi$ be the unique point at which $f_\xi$ is not differentiable, i.e. $t_\xi = \log\frac{1}{\|\xi - \eta\|}$. Then $f_\xi'(t) = 1$ for $t < t_\xi$ and $f_\xi'(t) = -1$ for $t > t_\xi$. Letting
\[
h_\xi(t) = f_\xi(t)(k_\xi - \delta) - \Psi(t),
\]
the assumption \eqref{Psiassumption3} guarantees that $h_\xi'(t) < 0$ for $t < t_\xi$ and $h_\xi'(t) > 0$ for $t > t_\xi$, assuming $t$ is sufficiently large. It follows that for all but finitely many $\xi\in G(P_{<\delta})$,
\[
\inf_{\substack{t > 0 \\ \eta_t\in H_\xi}} h_\xi(t)
= \inf_{t > 0} h_\xi(t)
= h_\xi(t_\xi) = \left[\left(t_\xi - \log\frac{1}{r_\xi}\right)(k_\xi - \delta) - \Psi(t_\xi)\right]
\]
and thus
\[
\log\underline D_\mu^\psi(\eta) \asymp_\plus \liminf_{\xi\in G(P_{<\delta})}\left[\left(\log\frac{1}{\|\xi - \eta\|} - \log\frac{1}{r_\xi}\right)(k_\xi - \delta) - \Psi\left(\log\frac{1}{\|\xi - \eta\|}\right)\right].
\]
Exponentiating gives
\[
\underline D_\mu^\psi(\eta)
\asymp_\times \min_{p\in P_{<\delta}} \liminf_{\xi\in G(p)} \left(\frac{1/\|\xi - \eta\|}{1/r_\xi} \psi_p(\|\xi - \eta\|) \right)^{k_p - \delta}
= \min_{p\in P_{<\delta}} \left(\limsup_{\xi\in G(p)} \frac{r_\xi}{\theta_p(\|\xi - \eta\|)} \right)^{k_p - \delta}.
\]
The calculation
\begin{align*}
\limsup_{\xi\in G(p)} \frac{r_\xi}{\theta_p(\|\xi - \eta\|)}
&= \sup\left\{\alpha > 0 : \frac{r_\xi}{\theta_p(\|\xi - \eta\|)} \geq \alpha \text{ for infinitely many $\xi\in G(p)$} \right\}\\
&= \sup\left\{\alpha > 0 : \|\xi - \eta\| \leq \theta_p^{-1}(r_\xi/\alpha) \text{ for infinitely many $\xi\in G(p)$} \right\}\\
&= \sup\left\{\alpha > 0 : \eta \in \Omega_p(\phi_{p,\alpha}) \right\}
\end{align*}
finishes the proof.
\end{proof}

Fix $p\in P_{<\delta}\cup P_{>\delta}$ and $\alpha > 0$, and let us determine whether the series
\[
\Sigma_{p,\alpha} = \sum_{n = 1}^\infty \phi_{p,\alpha}(\lambda^n)^{2\delta - k_p}
\]
occuring in Theorem \ref{theoremkhinchin} converges or diverges. For convenience write $\Delta_p = 2\delta - k_p$. We have\footnote{In the calculation below, an integration bound of $\clubsuit$ means that the precise integration bound is irrelevant.}
\begin{align*}
\Sigma_{p,\alpha} \asymp_\plus \int_\clubsuit^\infty \phi_{p,\alpha}(\lambda^t)^{\Delta_p} \;\dee t
&=_\pt \int_\clubsuit^\infty \frac{\theta_p^{-1}(\lambda^t/\alpha)^{\Delta_p}}{\lambda^{\Delta_p t}} \;\dee t\\
&\asymp_\times \int_\clubsuit^\infty \theta_p^{-1}(x^{-1/\Delta_p})^{\Delta_p} \;\dee x \note{letting $x = \alpha^{\Delta_p}\lambda^{-\Delta_p t}$}\\
&\asymp_\plus \int_0^\clubsuit \theta_p(x^{1/\Delta_p})^{-\Delta_p} \;\dee x \by{Lemma \ref{lemmainverse}}\\
&=_\pt \int_0^\clubsuit \psi_p(x^{1/\Delta_p})^{\Delta_p} \;\frac{\dee x}{x}\\
&\asymp_\times \int_\clubsuit^\infty \psi_p(e^{-t})^{\Delta_p} \;\dee t \note{letting $t = -\log(x^{1/\Delta_p})$}\\
&=_\pt \int_\clubsuit^\infty \exp\left(-\frac{2\delta - k_p}{k_p - \delta} \Psi(t)\right) \;\dee t\\
&\asymp_\plus \Sigma_p := \sum_{t = 1}^\infty \exp\left(-\frac{2\delta - k_p}{k_p - \delta} \Psi(t)\right).
\end{align*}
By Theorem \ref{theoremkhinchin}, for $\mu$-a.e. $\eta\in\Lambda$, for every $p\in P_{<\delta}\cup P_{>\delta}$, and for every rational $\alpha > 0$, we have $\eta\in \Omega_p(\phi_{p,\alpha})$ if and only if the series $\Sigma_{p,\alpha} \asymp_{\plus,\times} \Sigma_p$ diverges. Thus by Lemma \ref{lemma1}, for $\mu$-a.e. $\eta\in\Lambda$ we have
\begin{align} \label{aeeta1}
\begin{split}
\overline D_\mu^\psi(\eta) \asymp_\times \max_{p\in P_{>\delta}} \left(\begin{cases}
\infty & \Sigma_p = \infty\\
0 & \Sigma_p < \infty
\end{cases}\right)^{k_p - \delta}
&= \begin{cases}
\infty & \exists p\in P_{>\delta} \;\; \Sigma_p = \infty\\
0 & \forall p\in P_{>\delta} \;\;\Sigma_p < \infty
\end{cases}
= \begin{cases}
\infty & \text{\eqref{Hausdorffseries} diverges}\\
0 & \text{\eqref{Hausdorffseries} converges}
\end{cases}\\
\note{if $\delta < \kmax$ and $\Psi\to +\infty$}
\end{split}\\ \label{aeeta2}
\begin{split}
\underline D_\mu^\psi(\eta) \asymp_\times \min_{p\in P_{<\delta}} \left(\begin{cases}
\infty & \Sigma_p = \infty\\
0 & \Sigma_p < \infty
\end{cases}\right)^{k_p - \delta}
&= \begin{cases}
0 & \exists p\in P_{<\delta} \;\; \Sigma_p = \infty\\
\infty & \forall p\in P_{<\delta} \;\;\Sigma_p < \infty
\end{cases}
= \begin{cases}
0 & \text{\eqref{packingseries} diverges}\\
\infty & \text{\eqref{packingseries} converges}
\end{cases}\\
\note{if $\delta > \kmin$ and $\Psi\to -\infty$}
\end{split}
\end{align}
%Now by Theorem \ref{theoremkhinchin}, \eqref{aeeta1} and \eqref{aeeta2} hold for $\mu$-almost every $\eta$ in $\Lambda$.
Combining with Corollary \ref{corollaryRTT}, we see that Theorem \ref{theorem1} holds if the hypothesis $\Psi\to +\infty$ is added to case (i) and the hypothesis $\Psi\to -\infty$ is added to case (ii). But these extra hypotheses can be added without loss of generality, since if they fail, Theorem \ref{theorem1} can be deduced by comparing the measures $\HH^\psi$ and $\PP^\psi$ to the measures $\HH^\delta$ and $\PP^\delta$, respectively.

\section{Proofs of auxiliary facts}
The following lemmas are probably well-known to experts, so we have separated them from the main part of the proof for ease of exposition.

\begin{lemma}
\label{lemmainverse}
Let $f:(0,\infty)\to(0,\infty)$ be a decreasing homeomorphism. Then
\[
\int_0^\infty f(x) \;\dee x = \int_0^\infty f^{-1}(x) \;\dee x.
\]
\end{lemma}
\begin{proof}
These integrals are respectively equal to the Lebesgue measures of the sets
\begin{align*}
S_f &= \{(x,y)\in (0,\infty)^2 : y < f(x)\}\\
S_{f^{-1}} &= \{(x,y)\in (0,\infty)^2 : y < f^{-1}(x)\},
\end{align*}
which are conjugate to each other via the measure-preserving isometry $(x,y)\mapsto (y,x)$.
\end{proof}

\begin{lemma}
\label{lemmadetat}
Let $H_\xi = H(\xi,r_\xi)$ be a horoball not containing the origin, and fix $\eta\in\del\B^d\butnot\{\xi\}$ and $t > 0$ such that $\eta_t\in H_\xi$. Then
\[
b(\eta_t) \asymp_\plus \min\left(t - \log\frac{1}{r_\xi}, 2\log\frac{1}{\|\xi - \eta\|} - \log\frac{1}{r_\xi} - t\right).
\]
\end{lemma}
\begin{proof}
Let $\busemann_\xi$ denote the Busemann function of $\xi$, i.e.
\[
\busemann_\xi(y,z) = \lim_{x\to\xi} [\dist(x,y) - \dist(x,z)],
\]
where $\dist$ denotes the hyperbolic distance in $\B^d$. Then
\begin{equation}
\label{detat1}
b(\eta_t) = \busemann_\xi(0,\eta_t) - \dist(0,H_\xi) \asymp_\plus \busemann_\xi(0,\eta_t) - \log\frac{1}{r_\xi}\cdot
\end{equation}
To estimate $\busemann_\xi(0,\eta_t)$, we use the well-known fact that the incircle radius of a (possibly ideal) hyperbolic triangle is uniformly bounded. Specifically, when we consider the triangle $\Delta = \Delta(0,\eta,\xi)$ and let $\xmod\in\geo0\eta$, $\ymod\in\geo0\xi$, $\zmod\in\geo\xi\eta$ be the points of $\Delta$ tangent to the incircle, then $\dist(\xmod,\ymod) \asymp_\plus \dist(\xmod,\zmod) \asymp_\plus 0$. Here $\geo pq$ denotes the geodesic (or geodesic ray) connecting $p$ and $q$. Write $s = \dist(0,\xmod)$, and for each $u > 0$ let $\geo \zmod\eta_u$ denote the unique point on the geodesic $\geo \zmod\eta$ such that $\dist(\zmod,\geo \zmod\eta_u) = u$. Then
\[
\begin{cases}
\dist(\eta_t,\xi_t) \asymp_\plus 0 & \text{ if $t \leq s$}\\
\dist(\eta_t,\geo \zmod\eta_{t - s}) \asymp_\plus 0 & \text{ if $t \geq s$}
\end{cases}
\]
and thus
\begin{equation}
\label{detat2}
\busemann_\xi(0,\eta_t)
\begin{cases}
\asymp_\plus \busemann_\xi(0,\xi_t) = t & \text{ if $t \leq s$}\\
\asymp_\plus \busemann_\xi(0,\geo \zmod\eta_{t - s}) \asymp_\plus 2\lb \xi|\eta\rb_0 - t & \text{ if $t \geq s$}
\end{cases}.
\end{equation}
Here $\lb \xi|\eta\rb_0$ denotes the \emph{Gromov product}
\[
\lb \xi|\eta\rb_0 = \lim_{\substack{x\to \xi \\ y\to\eta}} \frac12[\dist(0,x) + \dist(0,y) - \dist(x,y)].
\]
Plugging $t = s$ into \eqref{detat2} gives $s \asymp_\plus \lb \xi|\eta\rb_0$, and thus
\[
\busemann_\xi(0,\eta_t) \asymp_\plus \min(t,2\lb \xi|\eta\rb_0 - t).
\]
Combining with \eqref{detat1} together with the well-known asymptotic
\[
\lb \xi|\eta\rb_0 \asymp_\plus \log\frac{1}{\|\xi - \eta\|}
\]
completes the proof.
\end{proof}

\begin{lemma}
\label{lemmaHPconformal}
Let $G$ be a Kleinian group, and let $\psi$ be a Hausdorff gauge function satisfying \eqref{Psiassumption2}. Then the measures $\nu_1 = \HH^\psi\given\Lambda$ and $\nu_2 = \PP^\psi\given\Lambda$ are $\delta$-conformal with respect to $G$.
\end{lemma}
\begin{proof}
Fix $\lambda > 0$, and note that \eqref{Psiassumption2} guarantees that
\begin{equation}
\label{HP1}
\lim_{r\searrow 0} \frac{\psi(\lambda r)}{\psi(r)} = \lambda^\delta.
\end{equation}
Now let $A\subset\Lambda$ be a Borel set on which $g\in G$ is $\lambda$-Lipschitz continuous. If $(U_i)_1^\infty$ is a countable cover of $A$ satisfying $\diam(U_i)\leq \epsilon\all i$, then $(\w U_i := g(U_i\cap A))_1^\infty$ is a cover of $g(A)$ satisfying $\diam(\w U_i) \leq \lambda\epsilon\all i$ and
\[
\sum_{i = 1}^\infty \psi(\diam(\w U_i)) \leq \left(\sup_{r\leq \epsilon}\frac{\psi(\lambda r)}{\psi(r)}\right)\sum_{i = 1}^\infty \psi(\diam(U_i)).
\]
Taking the infimum, letting $\epsilon\to 0$, and using \eqref{HP1} shows that $\nu_1(g(A)) \leq \lambda^\delta \nu_1(A)$. Similarly, if $(B_i = B(x_i,r_i))_1^\infty$ is a countable disjoint collection of balls with centers in $g(A)$ satisfying $\diam(B_i)\leq\epsilon\all i$, then $(\w B_i := B(g^{-1}(x_i),r_i/\lambda))_1^\infty$ is a disjoint\footnote{Here we use the fact that for all $x_1,x_2\in\R^d$ and $r_1,r_2 > 0$, $B(x_1,r_1)$ is disjoint from $B(x_2,r_2)$ if and only if $\dist(x_1,x_2)\geq r_1 + r_2$.} collection of balls with centers in $A$ satisfying $\diam(\w B_i) \leq \epsilon/\lambda \all i$ and
\[
\sum_{i = 1}^\infty \psi(\diam(B_i)) \leq \left(\sup_{r\leq \epsilon/\lambda}\frac{\psi(\lambda r)}{\psi(r)}\right)\sum_{i = 1}^\infty \psi(\diam(\w B_i)).
\]
Taking the supremum, letting $\epsilon\to 0$, and using \eqref{HP1} shows that $\nu_2(g(A)) \leq \lambda^\delta \nu_2(A)$. So
\[
\nu_i(g(A)) \leq \lambda^\delta \nu_i(A) \;\; (i = 1,2).
\]
Setting $\lambda = \sup_A |g'|$ and using the geometric mean value theorem gives
\begin{equation}
\label{HP2}
\nu_i(g(A)) \leq \sup_A |g'|^\delta \nu_i(A).
\end{equation}
Now let $B\subset\Lambda$ be an arbitrary Borel set, and fix $g\in G$. Fix $\epsilon > 0$, and let $\AA$ be a partition of $\Lambda$ such that $\sup_A |g'|/\inf_A |g'| \leq 1 + \epsilon$ for all $A\in \AA$. We have
\begin{align*}
\nu_i(g(B)) = \sum_{A\in\AA} \nu_i(g(B\cap A))
&\leq \sum_{A\in\AA} \sup_A |g'|^\delta \nu_i(B\cap A) \by{\eqref{HP2}}\\
&\leq (1 + \epsilon) \sum_{A\in\AA} \inf_A |g'|^\delta \nu_i(B\cap A) \leq (1 + \epsilon) \int_B |g'|^\delta \;\dee \nu_i. \noreason
\end{align*}
Letting $\epsilon$ tend to zero shows that $\nu_i(g(B)) \leq \int_B |g'|^\delta \;\dee\nu_i$. The reverse inequality is proved similarly.
\end{proof}

\bibliographystyle{amsplain}

\bibliography{bibliography}

\end{document}